\documentclass[12pt]{amsart}
\usepackage{amsthm}
\newtheorem{theorem}{Theorem}[section]

\newtheorem{proposition}[theorem]{Proposition}
\theoremstyle{definition}
\newtheorem{definition}[theorem]{Definition}

\newtheorem{example}[theorem]{Example}
\theoremstyle{remark}
\newtheorem{remark}[theorem]{Remark}

\counterwithin*{section}{part}

\usepackage{amssymb}
\usepackage{empheq}
\usepackage{stmaryrd}
\usepackage{enumerate}
\usepackage[shortlabels]{enumitem}
\usepackage{calc}
\usepackage{url}

\usepackage{mathabx}

\usepackage{xcolor}

\newcommand{\NN}{\mathbb{N}}
\newcommand{\ZZ}{\mathbb{Z}}
\newcommand{\RR}{\mathbb{R}}

\newcommand{\fluc}[3]{J_{{#1},{#2}}\{{#3}\}}

\newcommand{\flucinf}[2]{J_{{#1}}\{{#2}\}}

\newcommand{\norm}[1]{\lVert {#1} \rVert}

\newcommand{\seq}[1]{\{{#1}\}}

\usepackage[left=2.0cm,%
                right=2.0cm,%
                top=2.5cm,%
                bottom=3.5cm,%
                headheight=12pt,%
                a4paper]{geometry}%
\begin{document}
\title{The pointwise ergodic theorem on finitely additive spaces}
\author{Morenikeji Neri}
\maketitle
\date{\today}
\vspace*{-5mm}
\begin{center}
{\scriptsize 
Department of Mathematics, Technische Universit\"at Darmstadt\\
E-mail: neri@mathematik.tu-darmstadt.de}
\end{center}
\begin{abstract}
The almost sure convergence of ergodic averages in Birkhoff's pointwise ergodic theorem is known to fail in the finitely additive setting.

We introduce a natural reformulation of almost sure convergence suitable for finitely additive measures, which we call \emph{finite almost sure convergence}. Unlike the classical formulation, finite almost sure convergence only involves measures of finite unions and intersections, making it well adapted to finitely additive spaces.

Using this notion, we extend the pointwise ergodic theorem to finitely additive probability spaces. Our proof relies on demonstrating that several quantitative generalizations of the pointwise ergodic theorem remain valid in the finitely additive setting via an extension of the Calder\'on transference principle. The result then follows by exploiting the relationships between quantitative notions of almost sure convergence developed by the author and Powell (c.f.\ \emph{Trans.\ Amer.\ Math.\ Soc.\ Series B} \textbf{12} (2025), 974--1019).
\end{abstract}
\section{Introduction}\label{sec:intro}
Let $(\Omega, \mathcal{F}, \mu, \tau)$ be a dynamical system, where $(\Omega, \mathcal{F}, \mu)$ is a probability space and $\tau\colon \Omega \to \Omega$ is a measure-preserving automorphism. Birkhoff's pointwise ergodic theorem asserts that for any $f \in L^1(\Omega, \mathcal{F}, \mu)$, the ergodic averages
\[
A_n f(\omega) := \frac{1}{n} \sum_{i=1}^n f(\tau^i(\omega))
\]
converge almost surely to the integral $\int f \, d\mu$.

It is well known that this convergence may fail when $\mu$ is only finitely additive (see, for example,~\cite{ramakrishnan1986finitely}). However, in~\cite{ramakrishnan1986finitely}, Ramakrishnan shows that the pointwise ergodic theorem does hold in the finitely additive setting, provided that the sequence of functions $\seq{A_nf}$ satisfies a certain technical regularity condition with respect to $\mu$ (which is automatically satisfied in the $\sigma$-additive case), and that
\[
\mu\!\left(\bigcup_{i=1}^\infty A_i\right) = 0
\quad \text{whenever} \quad
\mu(A_i) = 0 \text{ for all } i,
\]
where $\seq{A_i}$ is a sequence of sets in the sub $\sigma$-algebra of $\tau$-invariant sets, that is, $\tau^{-1}(A_i) = A_i$ for all $i$.

Many other limit theorems in probability theory also fail in the finitely additive setting, such as the strong law of large numbers and the martingale convergence theorem. A substantial body of work, primarily from the 1970s and 1980s, is devoted to recovering classical results of probability theory in the finitely additive context by imposing additional structural or regularity assumptions on the space or the functions involved. In particular, motivated by the insights of Dubins and Savage in their book \emph{How to Gamble If You Must} \cite{dubins65gamble} and by the technical developments of Purves and Sudderth \cite{purves1976some}, there has been considerable interest \cite{chen1976some,chen1977almost,karandikar1982general,ramakrishnan1984central} in establishing various limit theorems over products of finitely additive probability spaces that possess suitable continuity properties.

In this paper, we argue that the almost sure convergence in the pointwise ergodic theorem actually \emph{does} hold for finitely additive probability spaces, provided one adopts the \emph{correct} notion of almost sure convergence in this setting.
\subsection{Motivation of main results}
Recall that a sequence of random variables \(\seq{X_n}\) on a probability space \((\Omega, \mathcal{F}, \mu)\) converges almost surely if
\begin{enumerate}
    \item[(A)]
    $\mu\big(\{\omega \in \Omega : \forall \varepsilon > 0\, \exists N\in \NN\, \forall i,j \ge N\, (|X_i(\omega) - X_j(\omega)| \le \varepsilon) \}\big) = 1.$
\end{enumerate}
By invoking the continuity of measure (see Lemma 3.1 of \cite{neri-powell:pp:martingale}), this definition is equivalent to each of the following formulations:
\begin{enumerate}
    \item[(B)] \(\forall \varepsilon > 0 \,\left(\mu\big(\{\omega \in \Omega : \exists N \in \NN\, \forall i,j \ge N\, (|X_i(\omega) - X_j(\omega)| \le \varepsilon) \}\big) = 1\right).\)\medskip

    \item[(C)] \(\forall \varepsilon, \lambda > 0\, \exists N \in \NN\, \left(\mu\big(\{\omega \in \Omega : \forall i,j \ge N\, (|X_i(\omega) - X_j(\omega)| \le \varepsilon) \}\big) > 1 - \lambda\right).\)\medskip

    \item[(D)] \(\forall \varepsilon, \lambda > 0\, \exists N \in \NN\, \forall k\in \NN\, \left(\mu\big(\{\omega \in \Omega : \forall i,j \in [N;N+k]\,
    (|X_i(\omega) - X_j(\omega)| \le \varepsilon) \}\big) > 1 - \lambda\right).\)
\end{enumerate}
Here, and throughout the rest of this article, we write $[a;b] := [a,b] \cap \mathbb{Z}$.

However, these equivalences break down if we assume that $\mu$ is only finitely additive. One can readily verify that in this case 
\[
(A) \Rightarrow (B), \qquad (C) \Rightarrow (B), \qquad \text{and} \qquad (C) \Rightarrow (D),
\]
but no other implication holds in general, as explicit counterexamples can be constructed for each of the remaining directions. 

We have already noted that the pointwise ergodic theorem does not hold when almost sure convergence is interpreted as \((A)\). We can strengthen this result by showing that the theorem also fails under \((B)\) (and hence under \((C)\)):

\begin{example}
\label{ex:density}
Recall that a set \(A \subseteq \ZZ\) is said to have \emph{asymptotic density} if the limit  
\[
d(A) := \lim_{N \to \infty} \underbrace{\frac{|A \cap [-N;N]|}{2N+1}}_{=: d_N(A)} 
\]
exists. It is straightforward to verify that \(d\) defines a finitely additive probability measure on the algebra of all subsets of \(\ZZ\) for which the above limit exists. Using a Banach limit, this measure can be extended to a finitely additive probability measure \(\mu\) on the entire power set \(\mathcal{P}(\ZZ)\) (c.f.\ \cite{agnew1938extensions}).  

Fix such a measure \(\mu\). Let \(f := I_A\) denote the indicator function of a set \(A \subseteq \ZZ^+\) that does not possess an asymptotic density, and let \(\tau(\omega) := \omega + 1\) denote the shift map, which can easily be shown to be a measurable, invertible, measure-preserving transformation of $\ZZ$. Since \(A\) does not have an asymptotic density, we can take \(\varepsilon > 0\) such that for all $N$ there exists \(i,j \ge N\) satisfying
\[
|d_i(A) - d_j(A)| > \frac{3\varepsilon}{2}.
\]
Now, for \(n \in \NN\) and $\omega \in \ZZ$ with $n\ge -\omega$ we have 
\begin{equation*}
    \begin{aligned}
        A_n f(\omega)
    &= \frac{1}{n}\sum_{i=1}^n f(\tau^i(\omega))
    =  \frac{|A \cap [\omega+1;\omega + n]|}{n}\\
    & = \frac{2(n+\omega)+1}{n} \frac{|A \cap [-(\omega + n);\omega + n]|}{2(n+\omega)+1}
      - \frac{|A \cap [1;\omega]|}{n}.
    \end{aligned}
\end{equation*}
Thus, if $\omega \le 0$ we have 
\[
A_n f(\omega)
    =2d_{n+\omega}(A)+ \frac{2\omega+1}{n}d_{n+\omega}(A)
\]
and if $\omega >0$ we have
\[
A_n f(\omega)
    = 2d_{n+\omega}(A)
      + \frac{2\omega+1}{n}\big(d_{n+\omega}(A) - d_{\omega}(A)\big).
\]
In both cases, since \(d_k(A) \in [0,1]\) for all \(k \in \NN\), it follows that
\[
|A_n f(\omega) - 2d_{n+\omega}(A)| \le \frac{2\omega+1}{n}.
\]
Fix \(\omega \in \NN\). Suppose there exists \(N \in \NN\) such that for all \(i,j \ge N\),
\[
|A_i f(\omega) - A_j f(\omega)| \le \varepsilon.
\]
If \(i,j \ge \max\{N, (2\omega+1)/\varepsilon\}\), the reverse triangle inequality gives
\[
|2d_{j+\omega}(A) - 2d_{i+\omega}(A)|
    - |A_i f(\omega) - 2d_{i+\omega}(A) + 2d_{j+\omega}(A) - A_j f(\omega)|
    \le |A_i f(\omega) - A_j f(\omega)|
    \le \varepsilon.
\]
Hence,
\[
2|d_{j+\omega}(A) - d_{i+\omega}(A)|
    \le \varepsilon + \frac{2\omega+1}{i} + \frac{2\omega+1}{j}
    \le 3\varepsilon,
\]
which contradicts the way we picked $\varepsilon$.

Therefore,
\[
\bigl\{
    \omega \in \Omega :
    \exists N\in \NN\, \forall i,j \ge N,\,
    |A_i f(\omega) - A_j f(\omega)| \le \varepsilon
\bigr\} = \emptyset,
\]
and so property \((B)\) is not satisfied for the sequence $\seq{A_nf}$.
\end{example}

Thus, the only formulation of almost sure convergence we are left with is \((D)\), which motivates the following definition:
\begin{definition}
    Let $\Omega$ be a set, $\mathcal{F}$ be an algebra of subsets of $\Omega$, and $\mu$ a finitely additive probability measure on $\mathcal{F}$. A sequence $\seq{X_n}$ of real valued functions on $\Omega$ is said to be \emph{finitely almost surely convergent} if 
\begin{equation}
\label{eq:measureability}
    \forall i,j \in \NN\, \forall \varepsilon>0\, \left(\{\omega \in \Omega: |X_i(\omega)-X_j(\omega)|\le \varepsilon\} \in \mathcal{F}\right)
\end{equation}
and
\[
\forall \varepsilon, \lambda>0\, \exists N \in \NN\, \forall k\in \NN\, \left(\mu( \forall i,j \in [N;N+k]\, (|X_i-X_j| \le \varepsilon)) > 1-\lambda\right).
\]
\end{definition}
\begin{remark}
If $\mathcal{F}$ is a $\sigma$-algebra and $\seq{X_n}$ are random variables (which is the setting considered in \cite{ramakrishnan1986finitely} and also the context of Example \ref{ex:density}), the measurability condition in (\ref{eq:measureability}) can be omitted. In general, this condition corresponds to the notion of \emph{weak Borel measurability} introduced in \cite{NeriPischke2023}, which represents the minimal assumption ensuring that the definition remains meaningful when $\mathcal{F}$ is merely an algebra.
\end{remark}
\subsection{Main results}
We shall show that finite almost sure convergence furnishes the \emph{appropriate} analogue of almost sure convergence for extending the pointwise ergodic theorem to finitely additive probability spaces. Our main result is the following:
\begin{theorem}
\label{thrm:main}
    Let $(\Omega, \mathcal{F},\mu)$ be a finitely additive probability space, $\tau: \Omega \to \Omega$ a measure-preserving automorphism, and $f \in L^1:=L^1(\Omega, \mathcal{F},\mu)$. If $\seq{A_nf}$ satisfies
\[
\forall i,j \in \NN\,\forall \varepsilon>0\, \left(\{\omega \in \Omega: |A_if(\omega)-A_jf(\omega)|\le \varepsilon\} \in \mathcal{F}\right)
\]
then $\seq{A_nf}$ is finitely almost surely convergent.
\end{theorem}
The proof of Theorem \ref{thrm:main} is quite subtle. First, we observe that the quantitative results on the fluctuations of ergodic averages established in \cite{jones:1998:oscillation} remain valid in the finitely additive setting. This is achieved by verifying that a suitable version of the Calder\'on transference principle \cite{calderon1968ergodic}, which we present in Theorem \ref{thrm:trans}, holds for finitely additive spaces. We then combine these results with the insights developed in \cite{neri-powell:pp:martingale} on the relationships between various quantitative formulations of almost sure convergence to obtain the desired conclusion. 

Furthermore, we obtain several quantitative strengthenings of Theorem \ref{thrm:main}, including a bound on the \emph{uniform metastability} of the ergodic avergaes (c.f.\ \cite{AVIGAD-GERHARDY-TOWSNER:10:Ergodic,neri-powell:pp:martingale,tao:07:softanalysis}):
\begin{theorem}
\label{thrm:mainmeta}
Let $(\Omega, \mathcal{F}, \mu)$ be a finitely additive probability space, $\tau: \Omega \to \Omega$ a measure-preserving automorphism, and $f \in L^1(\Omega, \mathcal{F}, \mu)$. Suppose that the sequence $\seq{A_n f}$ satisfies
\[
\forall i,j \in \NN\, \forall \varepsilon > 0\, \left(\{\omega \in \Omega : |A_i f(\omega) - A_j f(\omega)| \le \varepsilon\} \in \mathcal{F}\right).
\]
Then
\[
\forall \varepsilon, \lambda \in (0,1)\, \forall g:\NN \to \NN\, \exists n \le \tilde g^{(\lceil \delta(\lambda, \varepsilon) \rceil)}(0)\, \left(\mu\Big(\forall i,j \in [n, n + g(n)]\, (|A_i f - A_j f| \le \varepsilon)\Big) > 1 - \lambda\right),
\]
for some numerical constant $c$, where
\[
\delta(\lambda, \varepsilon) := \frac{c}{\lambda} \left( \frac{\|f\|_1}{\varepsilon \lambda} \right)^2,
\quad \text{and} \quad
\tilde g(n) := n + g(n),
\]
with $\tilde g^{(i)}$ denoting the $i$th iterate of $\tilde g$.
\end{theorem}
\subsection{Connections with mathematical logic and the proof mining program}
\label{sec:pm}
This work can be seen as a contribution to Kohlenbach's \emph{proof mining} program \cite{kohlenbach:08:book}. Proof mining employs tools from mathematical logic to analyze proofs in mainstream mathematics, yielding significant quantitative and qualitative improvements to known results.

There has been increasing interest in extending the proof mining program to probability theory. Recent developments include results on the laws of large numbers \cite{neri:kronecker:24,neri:quant:slln}, studies on the asymptotic behaviour of stochastic processes \cite{NeriPischkePowelllearn,neri-powell:pp:martingale}, and work on stochastic optimization \cite{NPP2025a,NeriPowell:RS:2024,PischkePowell:Halpern:2024}.

A key feature of analyzing proofs within the framework of proof mining is that one is often forced to strip a proof to its essential structure, which opens the door to generalizations.  The results in this paper also benefit from this aspect of proof mining. Studying convergence results in probability theory within this framework, particularly in light of the logical system introduced in \cite{NeriPischke2023}, reveals that $\sigma$-additivity is typically used only sparingly in many quantitative analyses. Consequently, many arguments can be lifted to the finitely additive setting. This was already observed in the investigation of quantitative aspects of the ergodic theorem in \cite{neri-powell:pp:martingale} and motivated the belief that an appropriate generalization of the pointwise ergodic theorem must exist for finitely additive probability spaces, which we present in this article.
\subsection{Organization of the rest of the paper}
The structure of the paper is as follows. In Section \ref{subsec:int}, we recall the theory of integration for finitely additive measures. Section \ref{subsec:probconv} reviews the relevant notions and results concerning quantitative formulations of almost sure convergence and their connection to the pointwise ergodic theorem, as presented in \cite{neri-powell:pp:martingale} and \cite{jones:1998:oscillation}. The proof of our main results are presented in Section \ref{sec:main}. 
\section{Preliminary definitions and lemmas}
\label{sec:prelim}
Throughout this section, fix a set $\Omega$, an algebra of subsets $\mathcal{F}$, and a finitely additive probability measure $\mu$.

\subsection{Integration on finitely additive spaces}
\label{subsec:int}
 We review the theory of integration on finitely additive probability spaces and recall the properties we shall need in this paper. Our exposition closely follows the seminal monograph by K. P. S. Bhaskara Rao and M. Bhaskara Rao \cite{RR1983}.

The first step is to define the integral of \emph{simple functions}.

\begin{definition}
A real-valued function $f$ on $\Omega$ is said to be \emph{simple} if there exist sets $\{A_i\}_{i=1}^n \subseteq \mathcal{F}$ forming a partition of $\Omega$ and real numbers $a_i \in \mathbb{R}$ such that, for all $\omega \in \Omega$,
\[
f(\omega) = \sum_{i=1}^n a_i I_{A_i}(\omega).
\]
For such a function $f$, the integral with respect to $\mu$ is defined by
\[
\mu(f) := \sum_{i=1}^n a_i \mu(A_i).
\]
\end{definition}
It is straightforward to verify that $\mu(f)$ is independent of the particular representation of $f$. Moreover, linear combinations of simple functions are simple, and the absolute value of a simple function is simple (see Propositions 4.4.2 and 4.4.4 in \cite{RR1983}).

As in the $\sigma$-additive case, a general \emph{integrable} function $f : \Omega \to \mathbb{R}$ will be defined as the limit of simple functions. In the $\sigma$-additive case, taking this limit to be almost surely, or in probability leads to equivalent definitions. In the finitely additive case, we adopt convergence in probability:

\begin{definition}[c.f.\ Definition 4.3.1 of \cite{RR1983}]
A sequence $\{f_n\}$ of real-valued functions on $\Omega$ is said to \emph{converge hazily} to a function $f$ if
\[
\forall \varepsilon > 0 \, \left(\lim_{n \to \infty} \mu^*\big(\{\omega \in \Omega : 
|f_n(\omega) - f(\omega)| > \varepsilon \}\big) = 0\right),
\]
where $\mu^*:\mathcal{P}(\Omega) \to [0,1]$ denotes the outer measure of $\mu$, defined by
\[
\mu^*(A) := \inf\{\mu(B) : A \subseteq B,\, B \in \mathcal{F}\}.
\]
\end{definition}

We can now define the integral for general functions.

\begin{definition}[c.f.\ Definition 4.4.11 of \cite{RR1983}]
A real-valued function $f$ on $\Omega$ is said to be \emph{integrable} if there exists a sequence $\{f_n\}$ of simple functions such that $\{f_n\}$ converges hazily to $f$ and
\[
\lim_{n,m \to \infty} \mu(|f_n - f_m|) = 0.
\]
Such a sequence $\{f_n\}$ is called a \emph{determining sequence} for $f$.  
The integral of $f$ is then defined by
\[
\mu(f) := \lim_{n \to \infty} \mu(f_n).
\]
Furthermore, we define
\[
L_1(\Omega, \mathcal{F}, \mu) := \{ f : \Omega \to \mathbb{R} \mid |f| \text{ is integrable} \}
\]
and write $\norm{f}_1$ for $\mu(\vert f\vert)$.
\end{definition}

The fact that the above definition of the integral is well-defined (that is, independent of the choice of determining sequence and finite) follows from Proposition 4.4.10 in \cite{RR1983}.

Furthermore, the finitely additive integral satisfies the expected fundamental structural properties (cf.\ Proposition 4.4.13 in \cite{RR1983}). In particular, linear combinations of integrable functions are themselves integrable, the integral defines a linear operator on the space of integrable functions, and it is monotone: if $f,g$ are integrable and $f(\omega) \le g(\omega)$ for all $\omega \in \Omega$, then $\mu(f) \le \mu(g)$.

The final result we require concerning integration on finitely additive probability spaces involves the composition of integrable functions with measure-preserving automorphisms. Specifically, we consider an invertible map $\tau \colon \Omega \to \Omega$ satisfying $\tau^{-1}(A) \in \mathcal{F}$ and $\mu(\tau^{-1}(A)) = \mu(A)$ for all $A \in \mathcal{F}$.

\begin{proposition}
If $f : \Omega \to \RR$ is integrable and $\tau : \Omega \to \Omega$ measure-preserving automorphism, then the function $\Tilde{f} : \Omega \to \RR$ defined by
\[
\Tilde{f}(\omega) := f(\tau(\omega))
\]
is integrable, and $\mu(\Tilde{f}) = \mu(f)$.
\end{proposition}

\begin{proof}
The result is immediate if $f$ is simple. Let $\seq{f_n}$ be a determining sequence for $f$, and define $\seq{\Tilde{f_n}}$ by
\[
\Tilde{f_n}(\omega) := f_n(\tau(\omega)).
\]
Since $\seq{f_n}$ and $\seq{\Tilde{f_n}}$ are sequences of simple functions and $\tau$ is measure preserving, we have
\[
\mu(f_n) = \mu(\Tilde{f_n}) \quad \text{for all } n,
\]
and hence, if $\seq{\Tilde{f_n}}$ is a determining sequence for $\Tilde{f}$, we have 
\[
\mu(\Tilde{f}) = \lim_{n \to \infty} \mu(\Tilde{f_n}) 
= \lim_{n \to \infty} \mu(f_n) = \mu(f).
\]
Thus, it remains to show that $\seq{\Tilde{f_n}}$ is indeed a determining sequence for $\Tilde{f}$.  

We first note the following.

\medskip
\noindent
\textbf{Claim.} If $\tau$ is a measure-preserving automorphism, then for every $E \subseteq \Omega$,
\[
\mu^*(\tau^{-1}(E)) = \mu^*(E).
\]
\emph{Proof of claim.}
For any $B \in \mathcal{F}$ with $E \subseteq B$, we have $\tau^{-1}(E) \subseteq \tau^{-1}(B)$ and
$\mu(\tau^{-1}(B)) = \mu(B)$.
Taking the infimum over all such $B$ yields
\[
\mu^*(\tau^{-1}(E)) 
\le \inf\{\mu(\tau^{-1}(B)) : E \subseteq B,\, B \in \mathcal{F}\}
= \mu^*(E).
\]
The reverse inequality follows by applying the same argument to $\tau^{-1}$.

\hfill$\blacksquare$
\medskip

Now, for all $\varepsilon > 0$ and $n \in \NN$,
\[
\mu^*\big(\{\omega : 
|\Tilde{f_n}(\omega) - \Tilde{f}(\omega)| > \varepsilon \}\big)
= \mu^*\big(\tau^{-1}(\{\omega : 
|f_n(\omega) - f(\omega)| > \varepsilon \})\big)
= \mu^*\big(\{\omega : 
|f_n(\omega) - f(\omega)| > \varepsilon \}\big).
\]
Thus, since $\seq{f_n}$ converges hazily to $f$, it follows that $\seq{\Tilde{f_n}}$ converges hazily to $\Tilde{f}$.

Finally, for all $n,m \in \NN$,
\[
\mu(|\Tilde{f_n} - \Tilde{f_m}|) 
= \mu(|f_n - f_m|),
\]
and hence
\[
\lim_{n,m \to \infty} \mu(|\Tilde{f_n} - \Tilde{f_m}|) 
= \lim_{n,m \to \infty} \mu(|f_n - f_m|) = 0.
\]
This shows that $\seq{\Tilde{f_n}}$ is indeed a determining sequence for $\Tilde{f}$, completing the proof.
\end{proof}

\subsection{Quantitative notions of stochastic convergence}
\label{subsec:probconv}
Quantitative aspects of the pointwise ergodic theorem have been studied extensively in the $\sigma$-additive setting. As we shall see in the following section, many of these quantitative results carry over to the finitely additive setting, since one can establish a version of the Calder\'on transference principle \cite{calderon1968ergodic} for finitely additive spaces. In this section, we review the relevant quantitative notions of almost sure convergence that are pertinent to our work.

The most natural quantitative interpretation of the convergence of a sequence of real numbers $\seq{x_n}$ is known as a \emph{rate of convergence}, which is a function $\phi:(0,1) \to \RR^+$ satisfying
\[
\forall \varepsilon \in (0,1)\, \forall i,j \ge \phi(\varepsilon)\,(|x_i - x_j| \le \varepsilon).
\]
There are many convergence results in analysis where one cannot obtain a uniform rate of convergence for the sequences satisfying the premises of these results; a notable example is the monotone convergence theorem, where no uniform rate of convergence exists for all monotone sequences in $[0,1]$, since such sequences can converge arbitrarily slowly. This remains true even for the class of computable sequences of rational numbers \cite{specker:49:sequence}.
.

In situations where a uniform rate of convergence cannot be obtained, one can often obtain uniform bounds on an equivalent reformulation, known as \emph{metastable convergence}:
\begin{equation}
\label{eqn:metastable}
\forall \varepsilon>0\, \forall g:\NN\to \NN \, \exists n \, \forall i,j\in [n;n+g(n)]\, (|x_i-x_j|\leq \varepsilon). 
\end{equation}
It is easy to see that \eqref{eqn:metastable} is equivalent to the convergence of $\seq{x_n}$. Indeed, if $\seq{x_n}$ does not converge, then there exists some $\varepsilon>0$ such that
\[
\forall n\, \exists i,j\ge n \, (|x_i-x_j|> \varepsilon),
\]
so we can take a function $g:\NN\to\NN$ bounding $i,j$ depending on $n$, i.e.
\[
\forall n\, \exists i,j \in [n;n+g(n)]\, (|x_i-x|> \varepsilon),
\]
which contradicts \eqref{eqn:metastable}. The converse implication is immediate.

A \emph{rate of metastability} is a functional $\Phi$ bounding $n$ in \eqref{eqn:metastable}, i.e.
\[
\forall \varepsilon \in (0,1)\, \forall g:\NN\to \NN\, \exists n \le \Phi(\varepsilon,g)\, \forall i,j\in [n;n+g(n)] \, (|x_i-x_j|\leq \varepsilon). 
\]

\begin{remark}
One can easily verify that if $\seq{x_n} \subseteq [0,1]$ is monotone, then 
\[
\Phi(\varepsilon,g) := \tilde g^{(\lceil 1/\varepsilon \rceil)}(0)
\]
is a rate of metastability for $\seq{x_n}$, where $\tilde g(n) := n + g(n)$ and $\tilde g^{(i)}$ denotes the $i$th iteration of $\tilde g$.
\end{remark}

\begin{remark}
The term \emph{metastability} originates from Tao \cite{tao:07:softanalysis}, where he discussed metastable versions of the monotone convergence principle and other ``infinitary'' statements. From the perspective of logic, metastability \eqref{eqn:metastable} corresponds to the Herbrand normal form of convergence, and a rate of metastability provides a solution to the so-called ``no-counterexample interpretation'' of the convergence statement \cite{kreisel:51:proofinterpretation:part1, kreisel:52:proofinterpretation:part2}. Furthermore, there are theorems from mathematical logic guaranteeing the extractibility of uniform rates of metastability for large classes of proofs. Such uniformity was required in the proof of the convergence of the ergodic averages given in \cite{tao:08:ergodic}, where a uniform quantitative version of the dominated convergence theorem was required. Indeed, extractions of rates of metastability by proof-theoretic methods are common results in proof mining (e.g.\ \cite{kohlenbach2012effective, kohlenbach2021finitary,NeriPowell2023, powell2023finitization}). Furthermore, as we shall see later, metastability is intimately related to the notion of fluctuations of sequences, a concept that has attracted considerable attention in ergodic theory and probability theory (see, for example, \cite{avigad2015oscillation, kachurovskii:96:convergence, kalikow1999fluctuations}).
\end{remark}
In the stochastic setting, the analogue of a rate of convergence is a function 
$\phi: (0,1) \times (0,1)  \to \RR^+$ satisfying
\[
\forall \varepsilon, \lambda \in (0,1)\, \forall k \in \NN\, \left(\mu\Big(\forall i,j \in [\phi(\lambda,\varepsilon);\phi(\lambda,\varepsilon)+k] \, (|X_i - X_j| \le \varepsilon) \Big) > 1 - \lambda\right),
\]
where $\seq{X_n}$ is a sequence of real-valued functions on $\Omega$ satisfying
\[
\forall i,j \in \NN\, \forall \varepsilon \in (0,1)\, \left(\{\omega \in \Omega: |X_i(\omega)-X_j(\omega)|\le \varepsilon\} \in \mathcal{F}\right).
\]
Such a function is sometimes called a \emph{rate of almost sure convergence}. In Theorem 5.1 of \cite{AVIGAD-GERHARDY-TOWSNER:10:Ergodic}, it is demonstrated that no computable such rate can exist for the convergence of the ($\sigma$-additive) pointwise ergodic theorem. Motivated by the deterministic case, the notion of \emph{uniform metastable convergence}\footnote{There is also a notion of pointwise metastable convergence (c.f.\ Remark \ref{rem:pointwise}).} was introduced in \cite{AVIGAD-GERHARDY-TOWSNER:10:Ergodic}:

\begin{definition}
Let $\seq{X_n}$ be a sequence of real-valued functions on $\Omega$ satisfying
\[
\forall i,j \in \NN\, \forall \varepsilon>0\, \left(\{\omega \in \Omega: |X_i(\omega)-X_j(\omega)|\le \varepsilon\} \in \mathcal{F}\right).
\]
We say that $\seq{X_n}$ is \emph{uniformly metastable} if
\[
\forall \varepsilon,\lambda>0\, \forall g:\NN\to \NN\, \exists n\in \NN\,\left( \mu\Big(\forall i,j \in [n,n+g(n)] \, (|X_i - X_j| \le \varepsilon) \Big) > 1 - \lambda\right).
\]
\end{definition}
\begin{remark}
In \cite{Avigad-Dean-Rute:Dominated:12}, uniform metastability was introduced under the name of \emph{$\lambda$-uniform $\varepsilon$-metastable convergence} and given in a slightly different form. The formulation we present here is from \cite{neri-powell:pp:martingale} (see the discussion after Definition 4.10 of \cite{neri-powell:pp:martingale}).
\end{remark}

As in the deterministic case, uniform metastability is equivalent to finite almost sure convergence.

\begin{proposition}
\label{prop:metaequivconv}
Let $\seq{X_n}$ be a sequence of real-valued functions on $\Omega$. Then $\seq{X_n}$ is uniformly metastable if and only if it is finitely almost surely convergent.
\end{proposition}
\begin{proof}
Suppose first that $\{X_n\}$ is finitely almost surely convergent.  
Then for every $\varepsilon, \lambda > 0$ there exists $N \in \mathbb{N}$ such that
\[
  \forall k \in \mathbb{N} \, \left(
  \mu\big(\forall i,j \in [N;N+k]\, (|X_i - X_j| \le \varepsilon)
  \big) > 1 - \lambda\right).
\]
Let $g \colon \mathbb{N} \to \mathbb{N}$ be arbitrary.  
Taking $k = g(N)$ in the above yields
\[
  \mu\big(\forall i,j \in [N;N+g(N)]\, |X_i - X_j| \le \varepsilon\big) > 1 - \lambda,
\]
which shows that $\seq{X_n}$ is uniformly metastable.

\medskip
Conversely, suppose $\seq{X_n}$ is not finitely almost surely convergent.  
Then there exist $\varepsilon, \lambda > 0$ such that for all $N \in \mathbb{N}$ there exists $k \in \mathbb{N}$ with
\[
\mu\big(\forall i,j \in [N;N+k]\, |X_i - X_j| \le \varepsilon\big) \le 1 - \lambda.
\]
Define $g(N)$ to be the least such $k$.  
Then for this choice of $g$, we have, for all $N\in \NN$
\[
  \mu\big( \forall i,j \in [N;N+g(N)]\, |X_i - X_j| \le \varepsilon\big) \le 1 - \lambda,
\]
which shows that $\seq{X_n}$ fails to be uniformly metastable.
\end{proof}

To prove Theorem \ref{thrm:main}, we shall demonstrate that the ergodic averages are uniformly metastable. Furthermore, we shall obtain a quantitative strengthening of this result through the construction of explicit so-called \emph{rates of uniform metastability}:

\begin{definition}
\label{def:unifmeta}
Let $\seq{X_n}$ be a sequence of real-valued functions on $\Omega$ satisfying
\[
\forall i,j \in \NN\, \forall \varepsilon>0\, \left(\{\omega \in \Omega: |X_i(\omega)-X_j(\omega)|\le \varepsilon\} \in \mathcal{F}\right).
\]
A functional $\Phi$ is called a \emph{rate of uniform metastability} for $\seq{X_n}$ if
\[
\forall \varepsilon,\lambda \in (0,1)\, \forall g:\NN\to \NN\, \exists n \le \Phi(\lambda,\varepsilon,g)\, \left(\mu\Big(\forall i,j \in [n,n+g(n)] \, (|X_i - X_j| \le \varepsilon) \Big) > 1 - \lambda\right).
\]
\end{definition}

In \cite{AVIGAD-GERHARDY-TOWSNER:10:Ergodic}, rates of uniform metastability were computed for the pointwise ergodic theorem for $L_2$ functions, with the treatment of the full theorem later given by the author and Powell in \cite{neri-powell:pp:martingale}. Moreover, in \cite{neri-powell:pp:martingale}, a stronger quantitative result is established by introducing the concept of \emph{uniform learnability}:

\begin{definition}
\label{def:learn}
Let $\seq{X_n}$ be a sequence of real-valued functions on $\Omega$ satisfying
\[
\forall i,j \in \NN\, \forall \varepsilon>0\,\left(\{\omega \in \Omega: |X_i(\omega)-X_j(\omega)|\le \varepsilon\} \in \mathcal{F}\right).
\]
A function $\phi:(0,1)\times (0,1)\to \RR^+$ is called a \emph{learnable rate of uniform convergence} if, for any sequences $a_0<b_0 \le a_1<b_1 \le \dots$, there exists $n \le \phi(\lambda,\varepsilon)$ such that
\[
\mu\Big(\exists i,j\in [a_n,b_n] \ |X_i-X_j|\geq \varepsilon \Big) \le \lambda.
\]
\end{definition}
It turns out that a learnable rate of uniform convergence corresponds to a rate of metastability of a particularly nice form:

\begin{proposition}[c.f. \cite{neri-powell:pp:martingale}]
\label{prop:metaequivlearn}
Let $\seq{X_n}$ be a sequence of real-valued functions on $\Omega$ satisfying
\[
\forall i,j \in \NN\, \forall \varepsilon>0\, \left(\{\omega \in \Omega: |X_i(\omega)-X_j(\omega)|\le \varepsilon\} \in \mathcal{F}\right).
\]
Then 
\[
\Phi(\lambda,\varepsilon,g) = \tilde g^{(\lceil \phi(\lambda,\varepsilon) \rceil)}(0) \quad \text{with } \tilde g(n) := n + g(n)
\]
is a rate of uniform metastability for $\seq{X_n}$ if and only if $\lceil \phi(\lambda,\varepsilon) \rceil$ is a learnable rate of uniform convergence.
\end{proposition}

This result is stated in the context of $\sigma$-additive spaces, immediately following Definition 4.15 in \cite{neri-powell:pp:martingale}, as an easy consequence of an abstract result given in Lemma 3.5 of \cite{neri-powell:pp:martingale}. The finitely additive case follows similarly as a direct consequence of the same lemma.

We shall prove a quantitatively stronger result than that stated in Theorem \ref{thrm:main} by computing a learnable rate of uniform convergence for the ergodic averages. Theorem \ref{thrm:main} will then follow from Propositions \ref{prop:metaequivconv} and \ref{prop:metaequivlearn}.

To compute a learnable rate of uniform convergence for the pointwise ergodic theorem, we require a result from \cite{jones:1998:oscillation} concerning the fluctuations of ergodic averages on $\sigma$-additive probability spaces. First, we introduce the notion of the number of $\varepsilon$-fluctuations of a sequence of real numbers.

\begin{definition}
Let $\seq{x_n}$ be a sequence of real numbers. We define $\fluc{N}{\varepsilon}{x_n}$ to be the number of $\varepsilon$-fluctuations that occur in the initial segment $\{x_0,\ldots,x_{N-1}\}$, i.e., the maximal $k \in \NN$ such that there exist indices
\[
i_1 < j_1 \le i_2 < j_2 \le \dots \le i_k < j_k < N
\]
with 
\[
|x_{i_l}-x_{j_l}| \ge \varepsilon \quad \text{for all } l=1,\dots,k.
\]
We write
\[
\flucinf{\varepsilon}{x_n} = \lim_{N\to\infty} \fluc{N}{\varepsilon}{x_n}
\]
for the total number of $\varepsilon$-fluctuations that occur in $\seq{x_n}$. Note that $\flucinf{\varepsilon}{x_n}$ could be infinite.
\end{definition}

\begin{remark}
It is clear that $\seq{x_n}$ converges if and only if 
\[
\forall \varepsilon > 0 \, \left(\flucinf{\varepsilon}{x_n} < \infty\right).
\]
\end{remark}

In \cite{jones:1998:oscillation}, Jones, Kaufman, Rosenblatt, and Wierdl obtained the following quantitative strengthening of the pointwise ergodic theorem:

\begin{theorem}
\label{thrm:jonesivanov}
Let $(\Omega, \mathcal{F},\mu)$ be a $\sigma$-additive probability space, and let $\tau:\Omega \to \Omega$ be a measure-preserving automorphism. For any $f \in L_1$ and $k, \varepsilon > 0$, we have
\[
\mu\Big(\flucinf{\varepsilon}{A_nf} \ge k\Big) \le \frac{C \norm{f}_1}{\varepsilon \sqrt{k}},
\]
for some numerical constant $C$.
\end{theorem}

This result was originally conjectured by Ivanov \cite{kachurovskii:96:convergence} (see also \cite{neri-powell:pp:martingale}\footnote{The authors were unaware of the bound of Jones, Kaufman, Rosenblatt, and Wierdl, and by improving the methods of Ivanov, were able to give an improvement to his bound, which was still weaker than that in \cite{jones:1998:oscillation}.}), who, using entirely different methods, was only able to show
\[
\mu\Big(\flucinf{\varepsilon}{A_nf} \ge k\Big) \le C \sqrt{\frac{\log(k)}{k}},
\]
where $C$ is a constant depending only on $\norm{f}_1/\varepsilon$.

By establishing a version of the Calder\'on transference principle \cite{calderon1968ergodic} for finitely additive spaces, we shall show that a version of Theorem \ref{thrm:jonesivanov} also holds in the finitely additive setting.

It is clear from Theorem \ref{thrm:jonesivanov} that the ergodic averages satisfy
\[
\forall \varepsilon > 0 \, \left(\flucinf{\varepsilon}{A_nf} < \infty \quad \text{almost surely}\right),
\]
and thus converge almost surely. However, as in the case of convergence, bounded fluctuations can have different, nonequivalent formulations in the finitely additive setting. If $\mathcal{F}$ is a $\sigma$-algebra and $\seq{X_n}$ is a sequence of random variables on $\Omega$ (so that for all $k$ and $\varepsilon$, $\fluc{k}{\varepsilon}{X_n}$ is itself a random variable), we have the following formulations of having finitely many fluctuations almost surely:
\begin{enumerate}
    \item[$(J1)$] 
    \(
    \mu\Big(\{\omega \in \Omega : \forall \varepsilon > 0\, \exists N \in \NN\, \forall k\in \NN\, \fluc{k}{\varepsilon}{X_n}(\omega) < N\}\Big) = 1.
    \)\medskip

    \item[$(J2)$] 
    \(
    \forall \varepsilon > 0\, \left(\mu\Big(\{\omega \in \Omega : \exists N\in \NN\, \forall k \in \NN\, (\fluc{k}{\varepsilon}{X_n}(\omega) < N)\}\Big) = 1\right).
    \)\medskip

    \item[$(J3)$] 
    \(
    \forall \varepsilon, \lambda > 0\, \exists N\in \NN\, \left(\mu\Big(\{\omega \in \Omega : \forall k\in \NN\, (\fluc{k}{\varepsilon}{X_n}(\omega) < N)\}\Big) > 1 - \lambda\right).
    \)\medskip

    \item[$(J4)$] 
    \(
    \forall \varepsilon, \lambda > 0\, \exists N\in \NN\, \forall k\in \NN\, \left(\mu\Big(\{\omega \in \Omega : \fluc{k}{\varepsilon}{X_n}(\omega) < N\}\Big) > 1 - \lambda\right).
    \)
\end{enumerate}

As in the case of convergence, we have
\[
(J1) \Rightarrow (J2), \qquad (J3) \Rightarrow (J2), \qquad \text{and} \qquad (J3) \Rightarrow (J4).
\]

Furthermore, it is clear that 
\[
(J2) \iff (B),
\]
where $(B)$ is the formulation of almost sure convergence given by:
\[
\forall \varepsilon > 0\, \left(
    \mu\big(\{\omega \in \Omega : \exists N \in \NN\, \forall i,j \ge N\,
    (|X_i(\omega) - X_j(\omega)| \le \varepsilon) \}\big) = 1\right).
\]
Therefore, Example \ref{ex:density} demonstrates that the pointwise ergodic theorem cannot satisfy $(J2)$ in the finitely additive setting and thus cannot satisfy $(J3)$. However, Theorem \ref{thrm:jonesivanov} implies formulation $(J3)$, so it cannot hold in the finitely additive setting. We shall show that the $(J4)$ formulation of Theorem \ref{thrm:jonesivanov} does hold via the Calder\'on transference principle \cite{calderon1968ergodic} for finitely additive spaces.

Theorem \ref{thrm:main} will follow from our $(J4)$ formulation of Theorem \ref{thrm:jonesivanov} by establishing a quantitative connection between bounds on the fluctuations and learnable rates of uniform convergence. We first introduce the notion of a \emph{modulus of finite fluctuations}, as in \cite{neri-powell:pp:martingale}, which quantitatively captures the $(J4)$ formulation:

\begin{definition}
\label{def:modfinfluc}
Let $\seq{X_n}$ be a sequence of real-valued functions on $\Omega$ satisfying
\[
\forall i,j \in \NN\, \forall \varepsilon>0\, \, \left(\{\omega \in \Omega: |X_i(\omega)-X_j(\omega)|\le \varepsilon\} \in \mathcal{F}\right).
\]
A function $\phi:(0,1)\times (0,1)\to \RR^+$ is a \emph{modulus of finite fluctuations} for $\seq{X_n}$ if
\[
\forall \varepsilon, \lambda \in (0,1)\, \forall k\in \NN \, \left(\mu\Big(\fluc{k}{\varepsilon}{X_n} < \phi(\varepsilon,\lambda)\Big) > 1 - \lambda\right).
\]
\end{definition}

\begin{remark}
This definition is well-posed because, if 
\[
\forall i,j \in \NN\, \forall \varepsilon>0\, \left( \{\omega \in \Omega: |X_i(\omega)-X_j(\omega)|\le \varepsilon\} \in \mathcal{F}\right).
\]
then
\[
\forall k \in \NN\, \forall \varepsilon, N > 0\, \left(\{\omega \in \Omega : \fluc{k}{\varepsilon}{X_n}(\omega) < N\} \in \mathcal{F}\right).
\]
\end{remark}

In \cite{neri-powell:pp:martingale}, the author and Powell examined how quantitative notions of almost sure convergence, including those introduced here, relate to one another. For example, from a learnable rate of uniform convergence, one can obtain a rate of uniform metastability. Furthermore, the authors abstracted away from almost sure convergence and instead focused on general logical formulas of the same quantifier complexity, yielding more general quantitative results.

At the time of writing \cite{neri-powell:pp:martingale} the quantitative relationship between moduli of finite fluctuations and learnable rates of uniform convergence in the $\sigma$-additive setting, was left as an open problem. Recently, Powell communicated to the author a partial solution to this problem by constructing a learnable rate of uniform convergence from a modulus of finite fluctuations. Powell's construction also lifts naturally to the finitely additive setting. Namely, one has the following result:

\begin{proposition}
\label{prop:abs:fluc:to:learn}
Let $\seq{X_n}$ be a sequence of real-valued functions on $\Omega$ satisfying
\[
\forall i,j \in \NN\, \forall \varepsilon>0\, \{\omega \in \Omega: |X_i(\omega)-X_j(\omega)|\le \varepsilon\} \in \mathcal{F}.
\]
Suppose $\seq{X_n}$ has a modulus of finite fluctuations $\phi$. Then $\seq{X_n}$ has a learnable rate of uniform convergence given by
\[
\psi(\lambda,\varepsilon) := \frac{2 \, \phi(\lambda/2,\varepsilon)}{\lambda}.
\]
\end{proposition}

\begin{proof}
Fix $\lambda, \varepsilon \in (0,1]$ and consider sequences $a_0 < b_0 \le a_1 < b_1 \le \dots$. For $k \in \NN$, define the event
\[
B_k := \{\omega \in \Omega : \fluc{k}{\varepsilon}{X_n}(\omega) \ge \phi(\lambda/2,\varepsilon)\} \in \mathcal{F}.
\]
By the definition of the modulus of finite fluctuations, $\mu(B_k) \le \lambda/2$ for all $k \in \NN$. For $a,b \in \NN$, set
\[
C(a,b) := \{\omega \in \Omega : \exists i,j \in [a,b], \ |X_i(\omega)-X_j(\omega)| > \varepsilon\} \in \mathcal{F}.
\]

Suppose, for a contradiction, that for all $n \le \psi(\lambda,\varepsilon)$ we have $\mu(C(a_n,b_n)) > \lambda$. Then, for all $n \le \psi(\lambda,\varepsilon)$,
\[
\mu(C(a_n,b_n) \cap B_k^c) = \mu(C(a_n,b_n)) - \mu(C(a_n,b_n) \cap B_k) > \lambda - \mu(B_k) \ge \frac{\lambda}{2},
\]
where we set $k := b_{\phi(\lambda/2,\varepsilon)} + 1$.  

Now observe that
\[
\sum_{n \le \psi(\lambda,\varepsilon)} I_{C(a_n,b_n) \cap B_k^c} < \phi(\lambda/2,\varepsilon),
\]
and thus
\[
\frac{(\psi(\lambda,\varepsilon) + 1)\lambda}{2} \le \sum_{n \le \psi(\lambda)} \mu(C(a_n,b_n) \cap B_k^c) = \mu\Bigg(\sum_{n \le \psi(\lambda)} I_{C(a_n,b_n) \cap B_k^c}\Bigg) < \phi(\lambda/2,\varepsilon),
\]
which is a contradiction. Here, the last inequality follows from the monotonicity of the integral.
\end{proof}

\begin{remark}
The optimality of this construction follows from Example 4.18 in \cite{neri-powell:pp:martingale}. The converse relationship (constructing a modulus of finite fluctuations from a learnable rate of uniform convergence) is currently an open problem.
\end{remark}

\section{Proof of main results}
\label{sec:main}
The proof of the main result follows by establishing a variation of the Calder\'on transference principle \cite{calderon1968ergodic} for finitely additive spaces.

We adopt the same set up for transference as in \cite{kosz2024sharp}. Suppose that for each $N \in \NN$ we have mappings $\mathcal{O}_N:\RR^N \to \RR^+$ and $C:\RR^+ \to \RR^+$ such that, for any $K \in \NN$, $f:[1,2K] \to \RR$, and $a>0$, we have
\[
\tag{$\dagger$}
|\{k \in [1;2K] : \mathcal{O}_K(\seq{\Tilde{A}_n f(k)}_{n \in [1,K]}) \ge a\}| \le C(a) \sum_{i=1}^{2K} |f(i)|,
\]
where 
\[
\Tilde{A}_n f(k) := \frac{1}{n} \sum_{i=1}^n f(k+i \bmod 2K).
\]
Our goal is to obtain such an inequality for an arbitrary finitely additive dynamical system.

\begin{theorem}
\label{thrm:trans}
Let $(\Omega, \mathcal{F}, \mu)$ be a finitely additive probability space, $\tau: \Omega \to \Omega$ a measure-preserving automorphism, and $f \in L^1(\Omega, \mathcal{F}, \mu)$. If the averages 
\[
A_n f(\omega) := \frac{1}{n} \sum_{i=1}^n f(\tau^i \omega)
\]
satisfy
\[
\tag{$\star$}
\forall K \in \NN\, \forall a>0\, \left(\{\omega \in \Omega : \mathcal{O}_K(\seq{A_n f(\omega)}_{n \in [1,K]}) \ge a\} \in \mathcal{F}\right),
\]
then for all $a>0$ and $K \in \NN$,
\[
\mu(\mathcal{O}_K(\seq{A_n f}_{n \in [1,K]}) \ge a) \le C(a) \|f\|_1.
\]
\end{theorem}

\begin{proof}
Fix $a>0$ and $K \in \NN$. For each $\omega \in \Omega$, define $f_\omega : [1,2K] \to \RR$ by
\[
f_\omega(k) := f(\tau^k(\omega)).
\]
Then, for $k,n \in [1,K]$,
\[
\Tilde{A}_n f_\omega(k) = \frac{1}{n} \sum_{i=1}^n f_\omega(k+i) = \frac{1}{n} \sum_{i=1}^n f(\tau^i(\tau^k(\omega))) = A_n f(\tau^k(\omega)).
\]
For $k \in [1,K]$, define
\[
E_k := \{\omega \in \Omega : \mathcal{O}_K(\seq{A_n f(\tau^k(\omega))}_{n \in [1,K]}) \ge a\} = \{\omega \in \Omega : \mathcal{O}_K(\seq{\Tilde{A}_n f_\omega(k)}_{n \in [1,K]}) \ge a\}.
\]
By $(\star)$, $E_0 \in \mathcal{F}$. Moreover, $\mu(E_k) = \mu(E_0)$ for all $k \in [1,K]$ since $\tau$ is measure-preserving and invertible and $E_k = (\tau^k)^{-1}(E_0)$.
Hence,
\[
\mu(E_0) = \frac{1}{2K} \sum_{k=1}^{2K} \mu(E_k) = \frac{1}{2K} \mu\Big(\sum_{k=1}^{2K} I_{E_k}\Big).
\]
For each $\omega \in \Omega$,
\[
\sum_{k=1}^{2K} I_{E_k}(\omega) = |\{k \in [1,2K] : \mathcal{O}_K(\seq{\Tilde{A}_n f_\omega(k)}_{n \in [1,K]}) \ge a\}| \le C(a) \sum_{k=1}^{2K} |f_\omega(k)| = C(a) \sum_{k=1}^{2K} |f(\tau^k(\omega))|.
\]
By the monotonicity of the integral, we conclude
\[
\mu(E_0) \le \frac{C(a)}{2K} \sum_{k=1}^{2K} \mu(|f(\tau^k)|) = C(a) \|f\|_1,
\]
and the result follows.
\end{proof}
We can now establish a version of Theorem \ref{thrm:jonesivanov} in the finitely additive setting.
\begin{theorem}
\label{thrm:finjones}
There exists a constant $C$ such that for any finitely additive probability space $(\Omega, \mathcal{F}, \mu)$, measure-preserving automorphism $\tau: \Omega \to \Omega$, and $f \in L^1(\Omega, \mathcal{F}, \mu)$, if
\[
\forall i,j \in \NN\, \forall \varepsilon>0\, \{\omega \in \Omega: |A_if(\omega)-A_jf(\omega)|\le \varepsilon\} \in \mathcal{F},
\]
then for all $a, \varepsilon > 0$ and $K \in \NN$,
\[
\mu(\fluc{K}{\varepsilon}{A_n f} \ge a) \le \frac{C \|f\|_1}{\varepsilon \sqrt{a}}.
\]
\end{theorem}

\begin{proof}
Define $\mathcal{O}_K(\seq{x_n}) := \varepsilon \sqrt{\fluc{K}{\varepsilon}{x_n}}$. Condition $(\star)$ is satisfied since the set in question is a finite union of finite intersections of sets of the form $\{\omega \in \Omega : |A_i f(\omega) - A_j f(\omega)| \le \varepsilon\}$.

Applying Theorem \ref{thrm:jonesivanov} to the discrete space $([1,2K], \mathcal{P}([1,2K]), \mu_K)$ with $\mu_K(A) := |A|/2K$ and $\tau_K(\omega) = \omega + 1 \bmod 2K$ gives $(\dagger)$. The result then follows by Theorem \ref{thrm:trans}.
\end{proof}
\begin{remark}
    As is the case for Calder\'on transference principle is the $\sigma$-additive setting, our Theorem \ref{thrm:trans} allows us to extend several weak-type~$(1,1)$ inequalities on the ergodic avergaes (in the context of further measureability assumptions in the case where $\mathcal{F}$ is only assumed to be an algebra). This includes: finitely additve analogues of the maximal ergodic theorem, upcrossing inequalities (see, for example, \cite{ivanov1996oscillations,bish67foundations}), and inequalities on the variation and oscillation seminorms (see, for example, \cite{jones:1998:oscillation}) to the finitely additive setting. 
\end{remark}

\begin{proof}[Proof of Theorems \ref{thrm:main} and \ref{thrm:mainmeta}]
Theorem \ref{thrm:finjones} implies that 
\[
\phi(\lambda, \varepsilon) := \left( \frac{C \|f\|_1}{\varepsilon \lambda} \right)^2
\]
is a modulus of finite fluctuations (c.f.\ Definition \ref{def:modfinfluc}) for the sequence $\seq{A_n f}$. Moreover, by Proposition \ref{prop:abs:fluc:to:learn}, we obtain that 
\[
\delta(\lambda, \varepsilon) := \frac{8}{\lambda} \left( \frac{C \|f\|_1}{\varepsilon \lambda} \right)^2
\]
constitutes a learnable rate of uniform convergence (c.f.\ Definition \ref{def:learn}) for $\seq{A_n f}$. Consequently, Proposition \ref{prop:metaequivlearn} yields that 
\[
\Phi(\lambda, \varepsilon, g) := \tilde g^{(\lceil \delta(\lambda, \varepsilon) \rceil)}(0),
\quad \text{where } \tilde g(n) := n + g(n),
\]
serves as a rate of uniform metastability (c.f.\ Definition \ref{def:unifmeta}) for $\seq{A_n f}$. This yields, Theorem \ref{thrm:mainmeta}. Theorem \ref{thrm:main} then follows from Proposition \ref{prop:metaequivconv}.
\end{proof}

\begin{remark}
In Theorem 7.9 of \cite{neri-powell:pp:martingale}, the asymptotically sharper bound 
\[
\Delta(\lambda, \varepsilon) := \left( \frac{c K}{\varepsilon \lambda} \right)^2
\]
is established as a learnable rate of uniform convergence for $\seq{A_n f}$ in the $\sigma$-additive setting, where $c > 0$ is a universal constant and $K := \max\{1, \|f\|_1\}$. 

The methods employed in \cite{neri-powell:pp:martingale} crucially rely on countable additivity and therefore do not extend directly to the finitely additive framework. Determining whether comparably sharp learnable rates of uniform convergence can be obtained for the finitely additive version of the pointwise ergodic theorem remains an open problem, which we leave for future investigation.
\end{remark}
\begin{remark}
\label{rem:pointwise}
There is an alternative way to obtain Theorem \ref{thrm:main} without appealing to Proposition \ref{prop:abs:fluc:to:learn}. 
From a modulus of finite fluctuations, one can derive what is known as a rate of \emph{pointwise} metastability convergence 
(c.f.\ Theorem 4.16 of \cite{neri-powell:pp:martingale}), that is, a functional $\Phi$ satisfying
\[
\forall \varepsilon, \lambda \in (0,1)\, \forall g:\NN \to \NN\, \mu\Big(\exists n \le \Phi(\lambda, \varepsilon, g)\, 
\forall i,j \in [n, n + g(n)]\, (|X_i - X_j| \le \varepsilon)\Big) > 1 - \lambda.
\]
The notion of \emph{pointwise} metastable convergence was introduced in \cite{Avigad-Dean-Rute:Dominated:12}, 
where the motivation was to develop a quantitative version of Tao's metastable dominated convergence theorem 
(c.f.\ Theorem A.2 of \cite{tao:08:ergodic}). 
Tao's theorem shows that one can obtain rates of metastability for the convergence of the integrals in the dominated convergence theorem, 
which depend only on a slightly stronger variant of a rate of \emph{pointwise} metastability for the random variables appearing in its premise.

The authors of \cite{Avigad-Dean-Rute:Dominated:12} observed that if one has a  rate of \emph{uniform} metastability 
for the random variables in the premise of the dominated convergence theorem, then one can easily obtain a rate of metastability 
for the corresponding integrals. 
Their main result (c.f.\ Theorem 3.1 of \cite{Avigad-Dean-Rute:Dominated:12})  explicitly constructs a rate of uniform metastability from a rate of pointwise metastability 
(which they term a quantitative version of Egorov's theorem) 
and uses this quantitative result to strengthen Tao's metastable dominated convergence theorem. 
However, the construction given in \cite{Avigad-Dean-Rute:Dominated:12} is highly complex and involves a form of recursion on trees 
known as \emph{bar recursion} (introduced in \cite{spector:62:barrecursion}). 
As a consequence, the resulting bounds exhibit an enormous blow-up in complexity. For example, as shown on p. 11 of \cite{Avigad-Dean-Rute:Dominated:12}, even comparatively simple input rates of pointwise metastability can yield tower-exponential bounds.

It was shown in \cite{NeriPischke2023} that the results of \cite{Avigad-Dean-Rute:Dominated:12} 
also hold in finitely additive probability spaces. 
Thus, one may use the modulus of finite fluctuations obtained in Theorem \ref{thrm:finjones} 
to derive a rate of pointwise metastability, 
and then appeal to the finitely additive adaptation of \cite{Avigad-Dean-Rute:Dominated:12} 
given in \cite{NeriPischke2023} to obtain Theorem \ref{thrm:main}, 
albeit at the cost of very weak quantitative bounds on the resulting rate of uniform metastability.
\end{remark}
\medskip

{\bf Acknowledgements.} We would like to express our gratitude to Pedro Pinto, Nicholas Pischke, and Thomas Powell for their valuable comments and stimulating discussions during the preparation of this work. We are especially indebted to Thomas Powell for communicating his construction, from which Proposition \ref{prop:abs:fluc:to:learn} was obtained by a straightforward modification.

\bibliographystyle{acm}

\begin{thebibliography}{10}

\bibitem{agnew1938extensions}
{\sc Agnew, R.~P., and Morse, A.~P.}
\newblock Extensions of linear functionals, with applications to limits, integrals, measures, and densities.
\newblock {\em Annals of Mathematics 39}, 1 (1938), 20--30.

\bibitem{Avigad-Dean-Rute:Dominated:12}
{\sc Avigad, J., Dean, E., and Rute, J.}
\newblock A metastable dominated convergence theorem.
\newblock {\em Journal of Logic and Analysis\/} (2012).

\bibitem{AVIGAD-GERHARDY-TOWSNER:10:Ergodic}
{\sc Avigad, J., Gerhardy, P., and Towsner, H.}
\newblock Local stability of ergodic averages.
\newblock {\em Transactions of the American Mathematical Society 362}, 1 (2010), 261--288.

\bibitem{avigad2015oscillation}
{\sc Avigad, J., and Rute, J.}
\newblock {Oscillation and the mean ergodic theorem for uniformly convex Banach spaces}.
\newblock {\em Ergodic theory and dynamical systems 35}, 4 (2015), 1009--1027.

\bibitem{RR1983}
{\sc {Bhaskara Rao}, K., and {Bhaskara Rao}, M.}
\newblock {\em {Theory of Charges}}, vol.~109 of {\em Pure and Applied Mathematics}.
\newblock Elsevier, 1983.

\bibitem{bish67foundations}
{\sc Bishop, E.}
\newblock {\em Foundations of Constructive Analysis}.
\newblock McGraw-Hill, 1967.

\bibitem{calderon1968ergodic}
{\sc Calder{\'o}n, A.~P.}
\newblock Ergodic theory and translation-invariant operators.
\newblock {\em Proceedings of the National Academy of Sciences 59}, 2 (1968), 349--353.

\bibitem{chen1976some}
{\sc Chen, R.}
\newblock Some finitely additive versions of the strong law of large numbers.
\newblock {\em Israel Journal of Mathematics 24}, 3 (1976), 244--259.

\bibitem{chen1977almost}
{\sc Chen, R.}
\newblock On almost sure convergence in a finitely additive setting.
\newblock {\em Zeitschrift f{\"u}r Wahrscheinlichkeitstheorie und verwandte Gebiete 37}, 4 (1977), 341--356.

\bibitem{dubins65gamble}
{\sc Dubins, L.~E., and Savage, L.~J.}
\newblock {\em How to gamble if you must: Inequalities for stochastic processes}.
\newblock McGraw-Hill, 1965.

\bibitem{ivanov1996oscillations}
{\sc Ivanov, V.~V.}
\newblock Oscillations of averages in the ergodic theorem.
\newblock In {\em Doklady Akademii Nauk\/} (1996), vol.~347, Russian Academy of Sciences, pp.~736--738.

\bibitem{jones:1998:oscillation}
{\sc Jones, R., Kaufman, R., Rosenblatt, J., and Wierdl, M.}
\newblock Oscillation in ergodic theory.
\newblock {\em Ergodic Theory and Dynamical Systems 18}, 4 (1998), 889--935.

\bibitem{kachurovskii:96:convergence}
{\sc Kachurovskii, A.}
\newblock The rate of convergence in ergodic theorems.
\newblock {\em Russian Mathematical Surveys 51}, 4 (1996), 653--703.

\bibitem{kalikow1999fluctuations}
{\sc Kalikow, S., and Weiss, B.}
\newblock Fluctuations of ergodic averages.
\newblock {\em Illinois Journal of Mathematics 43}, 3 (1999), 480--488.

\bibitem{karandikar1982general}
{\sc Karandikar, R.~L.}
\newblock A general principle for limit theorems in finitely additive probability.
\newblock {\em Transactions of the American Mathematical Society 273}, 2 (1982), 541--550.

\bibitem{kohlenbach:08:book}
{\sc Kohlenbach, U.}
\newblock {\em {Applied Proof Theory: Proof Interpretations and their Use in Mathematics}}.
\newblock Springer Monographs in Mathematics. Springer, 2008.

\bibitem{kohlenbach2012effective}
{\sc Kohlenbach, U., and Leu{\c{s}}tean, L.}
\newblock {Effective metastability of Halpern iterates in CAT(0) spaces}.
\newblock {\em Advances in Mathematics 231}, 5 (2012), 2526--2556.

\bibitem{kohlenbach2021finitary}
{\sc Kohlenbach, U., and Sipo{\c{s}}, A.}
\newblock The finitary content of sunny nonexpansive retractions.
\newblock {\em Communications in Contemporary Mathematics 23}, 01 (2021), 1950093.

\bibitem{kosz2024sharp}
{\sc Kosz, D.}
\newblock Sharp constants in inequalities admitting the {C}alder{\'o}n transference principle.
\newblock {\em Ergodic Theory and Dynamical Systems 44}, 6 (2024), 1597--1608.

\bibitem{kreisel:51:proofinterpretation:part1}
{\sc Kreisel, G.}
\newblock {On the Interpretation of Non-Finitist Proofs, {P}art {I}}.
\newblock {\em Journal of Symbolic Logic 16\/} (1951), 241--267.

\bibitem{kreisel:52:proofinterpretation:part2}
{\sc Kreisel, G.}
\newblock {On the Interpretation of Non-Finitist Proofs, {P}art {II}: Interpretation of Number Theory}.
\newblock {\em Journal of Symbolic Logic 17\/} (1952), 43--58.

\bibitem{neri:kronecker:24}
{\sc Neri, M.}
\newblock {A finitary Kronecker's lemma and large deviations in the Strong Law of Large Numbers on Banach spaces}.
\newblock {\em Annals of Pure and Applied Logic 176}, 6 (2025), 103569.

\bibitem{neri:quant:slln}
{\sc Neri, M.}
\newblock {Quantitative Strong Laws of Large Numbers}.
\newblock {\em Electronic Journal of Probability 30\/} (2025), 1--22.

\bibitem{NeriPischke2023}
{\sc Neri, M., and Pischke, N.}
\newblock Proof mining and probability theory.
\newblock {\em Forum of Mathematics, Sigma\/} (2025+).
\newblock To appear.

\bibitem{NeriPischkePowelllearn}
{\sc Neri, M., Pischke, N., and Powell, T.}
\newblock {Generalized learnability of stochastic principles}.
\newblock In {\em Proceedings of Computability in Europe (CiE 2025)\/} (2025), Lecture Notes in Computer Science, Springer.

\bibitem{NPP2025a}
{\sc Neri, M., Pischke, N., and Powell, T.}
\newblock On the asymptotic behaviour of stochastic processes, with applications to supermartingale convergence, dvoretzky's approximation theorem, and stochastic quasi-fejér monotonicity.
\newblock Preprint, available at \url{https://arxiv.org/abs/2504.12922}, 2025.

\bibitem{NeriPowell2023}
{\sc Neri, M., and Powell, T.}
\newblock A computational study of a class of recursive inequalities.
\newblock {\em Journal of Logic and Analysis 15}, 3 (2023), 1--48.

\bibitem{NeriPowell:RS:2024}
{\sc Neri, M., and Powell, T.}
\newblock {A quantitative Robbins-Siegmund theorem}.
\newblock {\em Annals of Applied Probability\/} (2025+).
\newblock To appear.

\bibitem{neri-powell:pp:martingale}
{\sc Neri, M., and Powell, T.}
\newblock {On quantitative convergence for stochastic processes: Crossings, fluctuations and martingales}.
\newblock {\em Transactions of the American Mathematical Society, Series B 12\/} (2025), 974--1019.

\bibitem{PischkePowell:Halpern:2024}
{\sc Pischke, N., and Powell, T.}
\newblock {Asymptotic regularity of a generalised stochastic Halpern scheme with applications}.
\newblock Preprint, available at \url{https://arxiv.org/abs/2411.04845}, 2024.

\bibitem{purves1976some}
{\sc Purves, R.~A., and Sudderth, W.~D.}
\newblock Some finitely additive probability.
\newblock {\em The Annals of Probability\/} (1976), 259--276.

\bibitem{ramakrishnan1984central}
{\sc Ramakrishnan, S.}
\newblock Central limit theorems in a finitely additive setting.
\newblock {\em Illinois Journal of Mathematics 28}, 1 (1984), 139--161.

\bibitem{ramakrishnan1986finitely}
{\sc Ramakrishnan, S.}
\newblock {A finitely additive generalization of Birkhoff’s ergodic theorem}.
\newblock {\em Proceedings of the American Mathematical Society 96}, 2 (1986), 299--305.

\bibitem{specker:49:sequence}
{\sc Specker, E.}
\newblock {{N}icht Konstruktiv Beweisbare {S}{\"a}tze der {A}nalysis}.
\newblock {\em Journal of Symbolic Logic 14\/} (1949), 145--158.

\bibitem{spector:62:barrecursion}
{\sc Spector, C.}
\newblock Provably recursive functionals of analysis: a consistency proof of analysis by an extension of principles in current intuitionistic mathematics.
\newblock In {\em Recursive Function Theory: Proc. Symposia in Pure Mathematics\/} (1962), F.~D.~E. Dekker, Ed., vol.~5, American Mathematical Society, pp.~1--27.

\bibitem{tao:07:softanalysis}
{\sc Tao, T.}
\newblock Soft analysis, hard analysis, and the finite convergence principle.
\newblock Essay posted 23 May 2007, 2007.
\newblock Appeared in: ‘T. Tao, Structure and Randomness: Pages from Year One of a Mathematical Blog. AMS, 298pp., 2008’.

\bibitem{tao:08:ergodic}
{\sc Tao, T.}
\newblock Norm convergence of multiple ergodic averages for commuting transformations.
\newblock {\em Ergodic Theory and Dynamical Systems 28\/} (2008), 657--688.

\bibitem{powell2023finitization}
{\sc T.Powell}.
\newblock {A finitization of Littlewood's Tauberian theorem and an application in Tauberian remainder theory}.
\newblock {\em Annals of Pure and Applied Logic 174}, 4 (2023), 103231.

\end{thebibliography}

\end{document}